\documentclass[preprint,12pt,3p,times]{elsarticle}

\usepackage{amsthm,amssymb,amsmath,amsfonts} 
\usepackage{eucal}      
\usepackage{mathrsfs}   

\def\Rset{\mathbb{R}}

\def\Nset{\mathbb{N}}
\def\Kset{\mathbb{K}}

\theoremstyle{plain}
\newtheorem{thm}{Theorem}[section]
\newtheorem{lem}[thm]{Lemma}
\newtheorem{cor}[thm]{Corollary}

\theoremstyle{definition}

\newtheorem{rem}[thm]{Remark}

\newtheorem{step}{Step}

\numberwithin{equation}{section}

\begin{document}

\begin{frontmatter}
\title{A general semilocal convergence theorem for simultaneous methods for polynomial zeros 
and its applications to Ehrlich's and Dochev-Byrnev's methods}
\author{Petko D. Proinov}\ead{proinov@uni-plovdiv.bg}
\address{Faculty of Mathematics and Informatics, University of Plovdiv, Plovdiv 4000, Bulgaria}

\begin{abstract}
In this paper, we establish a general semilocal convergence theorem (with computationally verifiable initial conditions and error estimates) 
for iterative methods for simultaneous approximation of polynomial zeros.
As application of this theorem, we provide new semilocal convergence results for Ehrlich's and Dochev-Byrnev's root-finding methods. 
These results improve the results of 
Petkovi\'c, Herceg and Ili\'c [Numer. Algorithms 17 (1998) 313--331] and
Proinov [C.~R. Acad. Bulg. Sci. 59 (2006) 705--712]. 
We also prove that Dochev-Byrnev's method (1964) is identical to Pre{\v s}i\'c-Tanabe's method (1972).
\end{abstract}


\end{frontmatter}



\section{Introduction and preliminaries}
\label{sec:Introduction-and-preliminaries}

Throughout the paper ${(\Kset,|\cdot|)}$ denotes a complete normed field 
and $\Kset[z]$ denotes the ring of polynomials over $\Kset$. 
The vector space $\Kset^n$ is endowed with the norm
${\|x\|_p = \left( \sum_{i=1}^n |x_i|^p \right)^{1/p}}$ for some ${1 \le p \le \infty}$, 
and the cone norm 
\[
{\|x\| = (|x_1|,\ldots,|x_n|)} 
\]
with values in $\Rset^n$. 
The vector space $\Rset^n$ is endowed with the standard coordinatewise orderings $\preceq$ and $\prec$
(see \cite[Example~5.1]{Pro13}). 

The present paper deals with the semilocal convergence of iterative methods for simultaneously finding all zeros of a polynomial
\begin{equation} \label{eq:polynomial-f}
	f(z) = C_0 z^n + C_1 z^{n-1} + \cdots + C_n
\end{equation}
in $\Kset[z]$ of degree ${n \ge 2}$. 
We define in $\Kset^n$ the \emph{vector of coefficients} of $f$ by 
\[
C_f = (C_1 / C_0,\ldots,C_n / C_0).
\]
A vector ${\xi \in \Kset^n}$ is called a \emph{root vector} of $f$ if $f$ can be presented in the form
\begin{equation} \label{eq:root-vector}
	f(z) = C_0 \prod_{i=1}^n (z - \xi_i).
\end{equation}
Obviously, a polynomial $f$ has a root vector ${\xi \in \Kset^n}$ if and only if it splits in $\Kset$.
Let ${V \colon \Kset^n \to \Kset^n}$ be the \emph{Vi\`ete function} defined by
${x \mapsto V(x)}$, where
$V(x)$ is the vector of coefficients of the monic polynomial ${\prod_{i=1}^n (z - x_i)}$.
In other words, the components of the vector ${V(x) \in \Kset^n}$ are the elementary symmetric polynomials defined by
\[
	V_i(x) = (-1)^i \sum_{1 \le j_1 < \ldots < j_i \le n}{x_{j_1} \ldots x_{j_i}} \qquad (i = 1,\ldots,n).
\]
It is well known that a vector ${\xi \in \Kset^n}$ is a root vector of $f$ if and only if it is a solution of Vi\`ete's system
\begin{equation} \label{eq:Viete-system}
V(x) = C_f \, .
\end{equation}


In 1891, taking into account this simple fact, Weierstrass \cite{Wei91} introduced and studied the first simultaneous method 
for polynomial zeros.
He provided a semilocal convergence analysis of his method without assuming that the system \eqref{eq:Viete-system} has a solution.
The \emph{Weierstrass method} is defined by 
\begin{equation}  \label{eq:Weierstrass-iteration}
x^{(k + 1)}  = x^{(k)} - W_f(x^{(k)}), \qquad k = 0,1,2,\ldots.
\end{equation}
Here and throughout, the Weierstrass correction ${W_f \colon \mathcal{D} \subset \Kset^n \to \Kset^n}$ is defined by
\begin{equation} \label{eq:Weierstrass-correction}
W_f(x) = (W_1(x),\ldots,W_n(x)) \quad\text{with}\quad
W_i(x) = \frac{f(x_i)}{C_0 \prod_{j \ne i} (x_i  - x_j)}
\quad (i = 1,\ldots,n),
\end{equation}
where $\mathcal{D}$ is the set of all vectors in $\Kset^n$ with pairwise distinct components.
Kerner \cite{Ker66} has proved that Weierstrass's method coincides with Newton's method 
in $\Kset^n$,
\begin{equation} \label{eq:Newton-iteration}
x^{(k+1)} = x^{(k)} - F'(x^{(k)})^{-1} F(x^{(k)}), \qquad k = 0,1,2,\ldots, 
\end{equation}
applied to Vi\`ete's system \eqref{eq:Viete-system}. In other words, Newton's method \eqref{eq:Newton-iteration} 
with ${F(x) = V(x) - C_f}$ and Weierstrass's method \eqref{eq:Weierstrass-iteration} are identical.
The detailed study of the convergence of the Weierstrass method can be found in Proinov \cite{Pro15a} and Proinov and Petkova \cite{PP14a}.


\medskip
In 1964, Dochev and Byrnev \cite{DB64} presented the second simultaneous method for polynomial zeros. 
The \emph{Dochev-Byrnev method} is defined by the following fixed point iteration: introduced
\begin{equation} \label{eq:Dochev-Byrnev-iteration}
	x^{(k + 1)} = \mathscr{F}(x^{(k)}), \qquad k = 0,1,2,\ldots,
\end{equation}
where the iteration function ${\mathscr{F} \colon \mathcal{D} \subset \Kset^n \to \Kset^n}$ is defined by
\begin{equation} \label{eq:Dochev-Byrnev-iteration-function}
\mathscr{F}(x) = (\mathscr{F}_1(x),\ldots,\mathscr{F}_n(x)) \quad\text{with}\quad
\mathscr{F}_i(x) = x_i  - \frac{f(x_i)}{g'(x_i)} \left( 2 - \frac{f'(x_i)}{g'(x_i)} + 
\frac{1}{2} \frac{f(x_i)}{g'(x_i)} \frac{g''(x_i)}{g'(x_i)} \right)
\end{equation}
and the polynomial $g$ is defined by 
\begin{equation} \label{eq:polynomial-g}
	g(z) = C_0 \prod_{i=1}^n (z - x_i).
\end{equation}
The local convergence of Dochev-Byrnev method \eqref{eq:Dochev-Byrnev-iteration} was studied by  
Semerdzhiev and Pateva \cite{SP77} and 
Kyurkchiev \cite{Kyu82} (see also \cite[Theorem~9.16]{SAK94}).


\medskip
In 1967, Ehrlich \cite{Ehr67} introduced and studied the third simultaneous method for polynomial zeros.
The \emph{Ehrlich method} is defined by the following fixed point iteration:
\begin{equation} \label{eq:Ehrlich-iteration}
	x^{(k + 1)} = \Phi(x^{(k)}), \qquad k = 0,1,2,\ldots,
\end{equation}
where the iteration function ${\Phi \colon \mathscr{D} \subset \Kset^n \to \Kset^n}$ is defined by
\begin{eqnarray} \label{eq:Ehrlich-iteration-function}
\Phi(x) & = & (\Phi_1(x),\ldots,\Phi_n(x)) \,\, \text{ with} \nonumber\\
\Phi_i(x) & = & x_i - \frac{f(x_i)}{f'(x_i) - f(x_i) \displaystyle\sum_{j \neq i}{\frac{1}{x_i - x_j}}} = 
            x_i  - \frac{W_i(x)}{1 + \displaystyle\sum_{j \ne i} \frac{W_j(x)}{x_i - x_j}} \, .
\end{eqnarray}
The first part of formula \eqref{eq:Ehrlich-iteration-function} is due to Ehrlich \cite{Ehr67} and 
the second one is due to B\"orsch-Supan \cite{Bor70}. 
To prove that the two parts of \eqref{eq:Ehrlich-iteration-function} are equal it is sufficient to substitute the derivative $f'(x_i)$ 
in \eqref{eq:Ehrlich-iteration-function} by
\begin{equation} \label{eq:f-derivative}
	f'(x_i) = \left(1 + \sum_{j \ne i} {\frac{W_i(x)}{x_i - x_j}} + \sum_{j \ne i} {\frac{W_j(x)}{x_i - x_j}} \right) 
	\prod_{j \ne i} (x_i - x_j).
\end{equation}
The equality \eqref{eq:f-derivative} can be found in Proinov and Cholakov \cite{PC14a}. Note that it holds for every 
monic polynomial ${f \in \Kset[z]}$ of degree ${n \ge 2}$, every vector ${x \in \mathcal{D}}$ and every 
${i \in I_n = \{1,2.\ldots,n\}}$. 
Obviously, the domain $\mathscr{D}$ of $\Phi$ is the set
\begin{eqnarray} \label{eq:Ehrlich-iteration-function-domain}
\mathscr{D} & = & 
\left\{ x \in \mathcal{D} : f'(x_i) - f(x_i) \sum_{j \neq i}{\frac{1}{x_i - x_j}} \ne 0 \, \text{ for all } \, i \in I_n \right\} \nonumber\\
 & = & \left\{ x \in \mathcal{D} : 1 + \sum_{j \ne i} \frac{W_j(x)}{x_i - x_j} \ne 0 \, \text{ for all } \, i \in I_n \right\}.
\end{eqnarray}
The detailed study of the local convergence of Ehrlich's method can be found in Proinov \cite{Pro15c}.
The semilocal convergence of Ehrlich's method was studied by 
Petkovi\'c \cite{Pet96}, 
Petkovi\'c and Ili\'c \cite{PI97},
Petkovi\'c and Herceg \cite{PH97,PH01}, 
Zheng and Huang \cite{ZH00} and
Proinov \cite{Pro06c}. 


\medskip
It is well known that ${(1 + t)^{-1} = 1 - t + o(t)}$ as ${t \to 0}$. 
Therefore, if $x \in \mathcal{D}$ is reasonably close to a root vector of $f$, then the following approximation is valid:
\begin{equation} \label{eq:approximation}
  \left(1 + \displaystyle\sum_{j \ne i} \frac{W_j(x)}{x_i - x_j}\right)^{-1} \approx 
	1 - \displaystyle\sum_{j \ne i} \frac{W_j(x)}{x_i - x_j} \, .
\end{equation}
Substituting \eqref{eq:approximation} in Ehrlich iteration function \eqref{eq:Ehrlich-iteration-function}, 
we obtain the following iterative method:
\begin{equation} \label{eq:Tanabe-iteration}
	x^{(k + 1)} = \mathscr{T}(x^{(k)}), \qquad k = 0,1,2,\ldots,
\end{equation}
where the iteration function ${\mathscr{T} \colon \mathcal{D} \subset \Kset^n \to \Kset^n}$ is defined by
\begin{equation} \label{eq:Tanabe-iteration-function}
\mathscr{T}(x) = (\mathscr{T}_1(x),\ldots,\mathscr{T}_n(x)) \quad\text{with}\quad
\mathscr{T}_i(x) = x_i  - W_i(x) \left( 1 - \sum_{j \ne i} \frac{W_j(x)}{x_i - x_j} \right).
\end{equation}
The method \eqref{eq:Tanabe-iteration} has been derived in 1972 by Pre{\v s}i\'c \cite{Pre72}. 
Two years later, Milovanovi\'c \cite{Mil74} gave an elegant derivation of this method. 
In 1983, the method \eqref{eq:Tanabe-iteration} was rediscovered by Tanabe \cite{Tan83}.
In the literature it is most frequently referred to as Tanabe's method.
In this paper, we refer to the method \eqref{eq:Tanabe-iteration} as the \emph{Pre{\v s}i\'c-Tanabe method}.
In 1996, Kanno, Kjurkchiev and Yamamoto \cite{KKY96} have proved that Pre{\v s}i\'c-Tanabe's method 
coincides with Chebyshev's method in $\Kset^n$, 
\begin{equation} \label{eq:Chebyshev-iteration}
x^{(k+1)} = x^{(k)} - \left[ I + \frac{1}{2} F'(x^{(k)})^{-1} F''(x^{(k)}) F'(x^{(k)})^{-1} F(x^{(k)}) \right] F'(x^{(k)})^{-1} F(x^{(k)}) 
\end{equation}
applied to Vi\`ete's system \eqref{eq:Viete-system} 
This means that Chebyshev's method \eqref{eq:Chebyshev-iteration} with ${F(x) = V(x) - C_f}$ 
is identical to Pre{\v s}i\'c-Tanabe's method \eqref{eq:Tanabe-iteration}.
For local convergence of Pre{\v s}i\'c-Tanabe's method, we refer to Toseva, Kyurkchiev and Iliev \cite{TKI11}. 
Semilocal convergence result of Pre{\v s}i\'c-Tanabe's method can be found in 
Petkovi\'c, Herceg and Ili\'c \cite{PHI97,PHI98} and
Ili\'c and Herceg \cite{IH98}. 

In this paper, we prove a general semilocal convergence result (Theorem~\ref{thm:general-semilocal-convergence-theorem-SM}) 
for simultaneous methods for polynomial zeros. 
Applying this result, we obtain new semilocal convergence theorems for Ehrlich's method 
(Theorem~\ref{thm:Ehrlich-semilocal})
and Dochev-Byrnev's method 
(Theorem~\ref{thm:Tanabe-semilocal}).
We also prove that Dochev-Byrnev's and Pre{\v s}i\'c-Tanabe's methods are identical 
(Theorem~\ref{thm:Dochev-Byrnev-Tanabe-equivalence}).
 
We note that the semilocal convergence of Ehrlich's and Dochev-Byrnev's methods can also be studied via local convergence results of these methods (see Proinov \cite{Pro15b} and Proinov and Vasileva \cite{PV15b}).
We will continue this topic in the future.

%
%

\section{A general semilocal convergence theorem for simultaneous methods}
\label{A-general-semilocal-theorem-for-simultaneous-methods}

In this section, we propose a semilocal convergence theorem for a class of iterative methods for simultaneous computation of all zeros of a polynomial. This result is based on our previous works \cite{Pro10,Pro15a}, where we presented semilocal convergence theorems for a large class of iterative methods in metric spaces.

In the sequel, for two vectors ${x \in \Kset^n}$ and ${y \in \Rset^n}$, we define in ${\Rset^n}$ the vector
\[
\frac{x}{y} = \left( \frac{|x_1|}{y_1},\ldots,\frac{|x_n|}{y_n} \right),
\]
provided that $y$ has no zero components. 
Also, we use the function ${d \colon \Kset^n \to \Rset^n}$ defined by 
\[
d(x) = (d_1(x),\ldots,d_n(x)) \quad\text{with}\quad 
d_i(x) = \min_{j \ne i} |x_i - x_j| \quad (i = 1,\ldots,n).
\]

Let ${f \in \Kset[z]}$ be a  polynomial of degree ${n \ge 2}$, and let ${T \colon D \subset \Kset^n \to \Kset^n}$ be an iteration function. 
We study the convergence of the iterative method
\begin{equation} \label{eq:simultaneous-method}
	x^{(k + 1)} = T(x^{(k)}), \qquad k = 0,1,2,\ldots,
\end{equation}
with respect to the function of initial conditions ${E_f \colon \mathcal{D} \to \Rset_+}$ defined by
\begin{equation} \label{eq:FIC3-SM}
E_f(x) = \left \| \frac{W_f(x)}{d(x)} \right \|_p \qquad (1 \le p \le \infty)
\end{equation}
and the convergence function ${F_f \colon \mathcal{D} \to \Rset_+^n}$ defined by
\begin{equation} \label{eq:CF}
F_f(x) =  \| W_f(x)\|,
\end{equation}
where ${W_f \colon \mathcal{D} \to \Kset^n}$ is the Weierstrass correction of $f$ defined by \eqref{eq:Weierstrass-correction}.

Throughout the paper, we denote by $J$ an interval on $\Rset_+$ containing $0$. 
For a given function ${\gamma \colon J \to \Rset_+}$, we define the functions ${\psi, \mu \colon J \to \Rset_+}$ by
\begin{equation} \label{eq:main-psi-mu}
\psi(t) = 1 - b t \gamma(t)  \quad\text{and}\quad \mu(t) = 1 - t \gamma(t),
\end{equation}
where ${b = 2^{1/q}}$. Here and throughout, for a given $p$ such that ${1 \le p \le \infty}$, 
we denote by $q$ the conjugate exponent of $p$, i.e. $q$ is defined by means of
\[
1 \le q \le \infty \quad\text{and}\quad 1/p + 1/q = 1.
\] 

For the sake of simplicity, throughout this section we use the following notations:
\begin{equation} \label{eq:sigma}
\sigma_i(x) = \sum_{j \ne i} {\frac{W_j(x)}{x_i - x_j}} \quad\text{and} \quad
\hat{\sigma}_i(x) = \sum_{j \ne i} {\frac{W_j(x)}{\hat{x}_i - x_j}} \, , 
\end{equation}
where ${x \in \mathcal{D}}$, ${\hat{x} = T(x)}$ and ${i \in I_n}$. 
It follows from the triangle inequality in $\Kset$, the definition of $d(x)$ and H\"older's inequality that
\begin{equation} \label{eq:sigma-estimate}
|\sigma_i(x)|  
\le  \sum_{j \ne i} \frac{|W_j(x)|}{|x_i - x_j|} 
\le  \sum_{j \ne i} \frac{|W_j(x)|}{d_j(x)} \le a E_f(x),
\end{equation}
where ${a = (n - 1)^{1/q}}$.
 
Before we state the main result of this section (Theorem~\ref{thm:general-semilocal-convergence-theorem-SM}), 
we first state a general theorem for iteration functions in $\Kset^n$. 
The purpose of this theorem is two-fold: 
(1) to be used in the proof of Theorem~\ref{thm:general-semilocal-convergence-theorem-SM}, and 
(2) to provide auxiliary results that can be used in the application of Theorem~\ref{thm:general-semilocal-convergence-theorem-SM}.

\begin{thm} \label{thm:iteration-function-properties}
Let ${f \in \Kset[z]}$ be a  polynomial of degree ${n \ge 2}$, ${T \colon D \subset \Kset^n \to \Kset^n}$ be an iteration function, 
and let ${1 \le p \le \infty}$.
Suppose ${J \subset \Rset_+}$ and ${x \in D}$ is a vector with distinct components such that ${E_f(x) \in J}$ and
\begin{equation} \label{eq:main-condition-gamma}
\|x - T(x)\| \preceq \gamma(E_f(x)) \, \|W_f(x)\|,
\end{equation}
where ${\gamma \colon J \to \Rset_+}$ is such that ${\psi \colon J \to \Rset_+}$ defined by \eqref{eq:main-psi-mu} 
is a positive function.
Then: 
\begin{enumerate}[(i)]
	\item $T(x) \in \mathcal{D}$;
	\item $|\hat{x}_i - \hat{x}_j| \ge \psi(E_f(x)) \, d_i(x)  > 0$, where $\hat{x}_i = T_i(x)$; 
	\item $|\hat{x}_i - x_j| \ge \mu(E_f(x)) \, d_i(x) > 0$, where ${\mu \colon J \to \Rset_+}$ is defined by \eqref{eq:main-psi-mu}; 
	\item $d(T(x)) \succeq \psi(E_f(x)) \, d(x) \succ 0$; 
	\item $\displaystyle|\hat{\sigma}_i(x)| \le \frac{a E_f(x)}{\mu(E_f(x))}$, where ${a = (n - 1)^{1/q}}$; 
	\item $\displaystyle|\sigma_i(x) - \hat{\sigma}_i(x)| \le \frac{a \gamma(E_f(x)) E_f(x)^2}{\mu(E_f(x))}$; 
	\item $\displaystyle\left|\prod_{j \ne i} \frac{\hat{x}_i - x_j}{\hat{x}_i - \hat{x}_j}\right| \le 
	\left( 1 + \frac{a E_f(x) \, \gamma(E_f(x))}{(n - 1) \psi(E_f(x))}\right)^{n-1}$. 
	\item If \, $\|W_f(T(x))\| \preceq \beta(E_f(x)) \, \|W_f(x)\|$, where ${\beta \colon J \to [0,1)}$, then
\begin{equation} \label{eq:FIC}
E_f(T(x)) \le \varphi(E_f(x)),
\end{equation}
where the function ${\varphi \colon J \to \Rset_+}$ is defined by ${\varphi(t) = t \beta(t) / \psi(t)}$.	
\end{enumerate}
\end{thm} 

\begin{proof}
Let $i \in I_n$ be fixed.
By Proposition~5.1 of \cite{Pro15a} with ${u = \hat{x}}$ and ${v = x}$, taking into account the definition of $d(x)$, 
we obtain for ${j \ne i}$ the inequalities:
\[
|\hat{x}_i - \hat{x}_j| \ge \left( 1 - b\left\| \frac{x - T(x)}{d(x)} \right\|_p \right) d_i(x) \quad\text{and}\quad
|\hat{x}_i - x_j| \ge \left( 1 - \left\| \frac{x - T(x)}{d(x)} \right\|_p \right) d_j(x).
\]
It follows from \eqref{eq:main-condition-gamma} that
\[
\left\| \frac{x - T(x)}{d(x)} \right\|_p \le E_f(x) \, \gamma(E_f(x)).
\]
Combining the last three inequalities, we obtain (ii) and (iii).
Note that (ii) implies (i).
Taking the minimum on both sides of (ii) over ${j \ne i}$, we get 
\[
{d_i(T(x)) \ge \psi(E_f(x)) d_i(x)} 
\]
which proves (iv).
It follows from the triangle inequality in $\Kset$, (iii) and H\"older's inequality that 
\begin{equation} \label{eq:sigma-hat-estimate}
|\hat{\sigma}_i(x)| 
\le  \sum_{j \ne i} \frac{|W_j(x)|}{|\hat{x} - x_j|} 
\le \frac{1}{\mu(E_f(x))} \sum_{j \ne i} \frac{|W_j(x)|}{d_j(x)} \le \frac{a E_f(x)}{\mu(E_f(x))}
\end{equation}
which proves (v).
For simplicity, we define the quantities $A_i(x)$ and $B_i(x)$ as follows
\begin{equation} \label{eq:A-B-definition}
A_i(x) = \sigma_i(x) - \hat{\sigma}_i(x) \quad\text{and}\quad 
B_i(x) = \prod_{j \ne i} {\frac{\hat{x}_i - x_j}{\hat{x}_i - \hat{x}_j}} \, . 
\end{equation}
From the triangle inequality, \eqref{eq:main-condition-gamma}, (iii), the definition of $d(x)$ and H\"older's inequality, we get
\begin{eqnarray} \label{eq:A-estimate}
&|A_i(x)|& = \left| \sum_{j \ne i} \frac{ W_j(x) (\hat{x}_i - x_i)}{(\hat{x}_i - x_j) (x_i - x_j)} \right|
        \le \sum_{j \ne i} \frac{|W_j(x)| \, |\hat{x}_i - x_i|}{|\hat{x}_i - x_j| \, |x_i - x_j|} \nonumber \\
&& \le \frac{\gamma(E_f(x)) \, |W_i(x)|}{\mu(E_f(x)) \, d_i(x)} \sum_{j \ne i} \frac{|W_j(x)|}{d_j(x)}
 \le \frac{a \, \gamma(E_f(x)) \, E_f(x)^2}{\mu(E_f(x))} 
\end{eqnarray}
which proves (vi).
It is easy to see that $B_i(x)$ can be written in the form
\[ 
	B_i(x) = \prod_{j \ne i} (1 + u_j), 
\]
where $u$ is a vector in $\Kset^{n - 1}$ with components ${u_j = (\hat{x}_j - x_j) / (\hat{x}_i - \hat{x}_j)}$ for ${j \ne i}$.
According to Proposition~5.5 of \cite{Pro15a}, we have the following estimate 
\begin{equation} \label{eq:B-estimate-u}
	|B_i(x)| \le \left( 1 + \frac{a \|u\|_p}{n - 1} \right)^{n - 1}.
\end{equation}
From \eqref{eq:main-condition-gamma} and (ii), we obtain
\[
	|u_j| = \frac{|\hat{x}_j - x_j|}{|\hat{x}_i - \hat{x}_j|} \le \frac{\gamma(E_f(x))}{\psi(E_f(x))} \, \frac{|W_j(x)|}{d_j(x)} \, .
\]
Taking the p-norm on both sides, we get
\begin{equation} \label{eq:u-estimate}
	\|u\|_p \le \frac{E_f(x) \, \gamma(E_f(x))}{\psi(E_f(x))} \, .
\end{equation}
Combining \eqref{eq:B-estimate-u} and \eqref{eq:u-estimate}, we deduce 
\begin{equation} \label{eq:B-estimate}
	|B_i(x)| \le \left( 1 + \frac{a E_f(x) \, \gamma(E_f(x))}{(n - 1) \psi(E_f(x))} \right)^{n-1}
\end{equation}
which proves (vii).
Now let 
\begin{equation} \label{eq:main-condition-beta}
\|W_f(T(x))\| \preceq \beta(E_f(x)) \, \|W_f(x)\|,
\end{equation}
where ${\beta \colon J \to [0,1)}$.
From \eqref{eq:main-condition-beta} and (iv), we deduce
\[
\left\| \frac{W_f(T(x))}{d(T(x))} \right\| \preceq
\frac{\beta(E_f(x))}{\psi(E_f(x))} \left\| \frac{W_f(x)}{d(x)} \right\|.
\]
Taking the $p$-norm on both sides, we obtain \eqref{eq:FIC}.
This completes the proof of the theorem.
\end{proof}

Now we are ready to state the main theorem of this paper. 
In this theorem, ${S_k(t)}$ stands for the polynomial
\(
S_k(t) = 1 + t + \cdots + t^{k - 1}
\)
if ${k \in \Nset}$, and ${S_0(t) \equiv 0}$.

\begin{thm} \label{thm:general-semilocal-convergence-theorem-SM}
Let ${(\Kset,|\cdot|)}$ be a complete normed field,
$f \in \Kset[z]$ be a polynomial of degree ${n \ge 2}$,
${T \colon D \subset \Kset^n \to \Kset^n}$ be an iteration function,  
${E_f \colon \mathcal{D} \to R_+}$ be defined by \eqref{eq:FIC3-SM}, and ${1 \le p \le \infty}$.
Suppose there exists ${\tau > 0}$ such that for any ${x \in \mathcal{D}}$ with ${E_f(x) < \tau}$, we have:
\begin{enumerate}[(a)]
\item $x \in D$,
\item $\|x - T(x)\| \preceq \gamma(E_f(x)) \, \|W_f(x)\|$,
\item $\|W_f(T(x))\| \preceq \beta(E_f(x)) \, \|W_f(x)\|$,
\end{enumerate}
where $\beta$ is a quasi-homogeneous function of degree ${m \ge 0}$ on ${[0,\tau)}$,
$\gamma$ is a nondecreasing positive function on ${[0,\tau)}$ such that
the function $\psi$ defined by \eqref{eq:main-psi-mu} is positive on ${[0,\tau)}$.
Let ${x^{(0)} \in \mathcal{D}}$ be such that
\begin{equation}
	E_f(x^{(0)}) < \tau \quad\text{and}\quad \phi(E_f(x^{(0)})) \le 1,
\end{equation}
where $\phi = \beta / \psi$.
Then the following statements hold true:
\begin{enumerate}[(i)]
\item 
\textsc{Convergence}. 
Starting from $x^{(0)}$, the Picard iteration \eqref{eq:simultaneous-method} is well-defined, remains in the closed ball 
${\overline{U}(x^{(0)},\rho)}$ and converges to a root vector $\xi$ of $f$, where 
\[
\rho =  \frac{\gamma(E_f(x^{(0)}))}{1 - \beta(E_f(x^{(0)}))} \, \| W_f(x^{(0)}) \|. 
\]
Besides, the convergence is of order ${r = m + 1}$ provided that ${\phi(E_f(x^{(0)})) < 1}$.
\item 
\textsc{A priori estimate}. For all $k \ge 0$ we have the following error estimate
\begin{equation}  \label{eq:semilocal-priori}
\left \| x^{(k)} - \xi \right \| \preceq A_k \, \frac{{\theta}^k {\lambda }^{S_k(r)}}
{1 - \theta {\lambda}^{r^k } } \left \| W_f(x^{(0)}) \right \| ,
\end{equation}
where $\lambda  = \phi(E_f(x^{(0)}))$, $\theta  = \psi(E_f(x^{(0)}))$ and
$A_k = \gamma(E_f(x^{(0)}) \lambda^{S_k(r)})$.
\item 
\textsc{First a posteriori estimate}. For all $k \ge 0$ we have the following error estimate
\begin{equation}  \label{eq:semilocal-posteriori-1}
\left\| x^{(k)} - \xi \right\| \preceq \frac{\gamma(E_f(x^{(k)}))}{1 - \beta(E_f(x^{(k)}))} \, \left\| W_f(x^{(k)}) \right\|.
\end{equation}
\item 
\textsc{Second a posteriori estimate}. For all $k \ge 0$ we have the following error estimate
\begin{equation}  \label{eq:semilocal-posteriori-2}
\left\| x^{(k + 1)} - \xi \right\| \preceq 
\frac{\theta_k \lambda_k}{1 - \theta_k (\lambda_k)^r} \, \gamma(E_f(x^{(k+1)})) \left\| W_f(x^{(k)}) \right\| ,
\end{equation}
where $\lambda_k = \phi(E_f(x^{(k)}))$ and  $\theta_k = \psi(E_f(x^{(k)}))$.
\item 
\textsc{Some other estimates}. For all ${k \ge 0}$ we have 
\begin{equation}  \label{eq:some-other-estimates}
\left\| W_f(x^{(k + 1)}) \right\| \preceq \theta \, \lambda^{r^k} \, \left\|W_f(x^{(k)}) \right\| \quad\text{and}\quad
\left\|W_f(x^{(k)}) \right\| \preceq \theta^k \, \lambda^{S_k(r)} \, \left\|W_f(x^{(0)}) \right\|.
\end{equation}
\item \textsc{Localization of the zeros}.
If ${E_f(x^{(0)}) < 1}$, then $f$ has $n$ simple zeros in $\Kset$ and for every ${k \ge 0}$ the closed disks
\begin{equation}  \label{eq:disks-disjoint}
D_i^k = \{z \in \Kset : |z-x_i^{(k)}| \le r_i^{(k)} \},
\quad i = 1,2,\ldots,n,
\end{equation}
where 
\[
r_i^{(k)} = \frac{\gamma(E_f(x^{(k)}))}{1 - \beta(E_f(x^{(k)}))} \, |W_i(x^{(k)})|,
\]
are mutually disjoint and each of them contains exactly one zero of $f$.	
\end{enumerate}
\end{thm} 

\begin{proof}
Without loss of generality we may assume that $f$ is monic.
The case ${E_f(x^{(0)}) = 0}$ is almost trivial. Indeed, in this case ${W_f(x^{(0)}) = 0}$. 
Hence, setting ${z = x^{(0)}}$ in Proposition 8.1 of \cite{Pro15a}, we obtain \eqref{eq:root-vector}
which means that ${x^{(0)}}$ is a root vector of $f$.
Using conditions (a) and (b), it can be proved by induction that ${x^{(k)} = x^{(0)}}$ for all ${k \ge 0}$. 
This proves that the theorem holds with ${\xi = x^{(0)}}$.

Now we shall consider the case ${E_f(x^{(0)}) > 0}$. 
In this case, we set ${R = E_f(x^{(0)})}$ and ${J = [0,R]}$.
We divide the proof into four steps.

\begin{step}
In this step, we consider some properties of the functions $\psi$, $\beta$, $\phi$ and $\varphi$, 
where $\varphi$ is defined by 
\(
{\varphi(t) = t \phi(t)}.
\)
The monotonicity and positivity of $\gamma$ imply that $\psi$ is decreasing on ${[0,\tau)}$ and ${\psi(t) \le 1}$ 
with equality if and only if ${t = 0}$.
Then $\phi$ is nondecreasing and nonnegative on ${[0,\tau)}$.  It is easy to see that for all ${t \in J}$,
\begin{equation}  \label{eq:phi-beta-inequalities}
0 \le \phi(t) \le 1, \quad \beta(t) \le \psi(t) \quad\text{and}\quad 0 \le \beta(t) < 1.
\end{equation}
To prove that ${\beta(t) < 1}$, it suffices to show that ${\beta(R) < 1}$.
From ${\beta(R) \le \psi(R) \le 1}$, we conclude that ${\beta(R) \le 1}$. 
The case ${\beta(R) = 1}$ is impossible since it implies ${\psi(R) = 1}$ which is a contradiction.
Consequently, ${\beta(t) < 1}$ on $J$. 
From \eqref{eq:phi-beta-inequalities} and the fact that $\beta$ is a quasi-homogeneous of degree ${m}$, we deduce that ${t \beta(t)}$  
is a strict gauge function of order $r$ on $J$, and that ${\varphi}$ is a gauge function of order $r$ on $J$. 

In particular, it follows from the properties of the considered functions that
\begin{equation} \label{eq:lambda-theta}
	0 \le \lambda \le 1, \quad 0 < \theta \le 1 \quad\text{and}\quad \lambda \, \theta < 1.
\end{equation}
\end{step}

\begin{step}
In this step we prove that the iteration \eqref{eq:simultaneous-method} 
converges to a point ${\xi \in \Kset^n}$ and that conclusions (i)-(v) hold, 
except the fact that $\xi$ is a root vector of $f$.
We prove this by applying Theorem~4.2 of \cite{Pro15a} to the iteration function ${T \colon D \subset \Kset^n \to \Kset^n}$.

Let ${x \in D}$ be such that ${{E_f(x) \in J}}$. 
It follows from (b), (c) and Theorem~\ref{thm:iteration-function-properties} that \eqref{eq:FIC} holds.
Hence, the function ${E_f \colon D \to \Rset_+}$ defined by \eqref{eq:FIC3-SM} is a function of initial conditions of $T$
with a gauge function $\varphi$ of order $r$ on the interval $J$.
It follows from (c) that the function ${F_f \colon D \to \Rset_+^n}$ defined by 
\eqref{eq:CF} is a convergence function of $T$ with control functions $\beta$ and $\gamma$. 

Now we shall prove that ${x^{(0)}}$ is an initial point of $T$.
It follows from (a) that ${x^{(0)} \in D}$.
According to Proposition~2.7 of \cite{Pro15a}, to prove that ${x^{(0)}}$ is an initial point of $T$,
it is sufficient to show that
\begin{equation} \label{eq:initial-point-test}
	x \in D \, \text{ and } \, E_f(x) \in J \, \text{ imply } \, T(x) \in D.
\end{equation}
From Theorem~\ref{thm:iteration-function-properties}(i), we get that ${T(x) \in \mathcal{D}}$.
From \eqref{eq:FIC} and ${E_f(x) \in J}$, taking into account that $\varphi$ is a gauge function of order ${r \ge 1}$ on $J$, 
we deduce that ${E_f(T(x)) \in J}$.
Thus we have both ${T(x) \in \mathcal{D}}$ and ${E_f(T(x)) \in J}$.
Then it follows from (a) that ${T(x) \in D}$ which proves \eqref{eq:initial-point-test}.
Hence, $x^{(0)}$ is an initial point of $T$.

Now it follows from Theorem~4.2 of \cite{Pro15a} that the Picard iteration \eqref{eq:simultaneous-method} is well-defined, lies in the ball ${\overline{U}(x^{(0)},\rho)}$ and converges to a vector $\xi \in \Kset^n$ with estimates 
\eqref{eq:semilocal-priori}, 
\eqref{eq:semilocal-posteriori-1},
\eqref{eq:semilocal-posteriori-2} and 
\eqref{eq:some-other-estimates}.
\end{step}

\begin{step}
In this step we prove that $\xi$ is a root vector of $f$.
In the previous step we have proved that the iterative sequence ${(x^{(k)})}$ defined by \eqref{eq:simultaneous-method} converges to $\xi$.
On the other hand, it follows from \eqref{eq:some-other-estimates} and \eqref{eq:lambda-theta} that the sequence ${(W_f(x^{(k)})})$ converges to the zero-vector in $\Kset^n$. 
By  Proposition~8.2 of \cite{Pro15a}, we obtain that $\xi$ is a root vector of $f$.
\end{step}

\begin{step}
In this step we prove claim (vi). 
Let ${\phi(E(x^{(0)})) < 1}$ and ${k \ge 0}$ be fixed. 
According to \eqref{eq:FIC}, ${E_f(x^{(k)}) \le E_f(x^{(0)})}$ which means that ${E_f(x^{(k)}) \in J}$.
Besides, ${\phi(E(x^{(0)})) < 1}$ implies that ${\phi(t) < 1}$ on $J$.
This yields ${\beta(t) < \psi(t)}$ for all ${t \in J}$, which can be written in the following equivalent form:
\[   
\beta(t) < 1 - b \, t \, \gamma(t) \quad\text{for all}\quad t \in J.
\]
Then it follows from Proposition~8.4 of \cite{Pro15a} that the disks \eqref{eq:disks-disjoint} are mutually disjoint. 
On the other hand, it follows from claim (iii) that each of these disks contains at least one zero of $f$.
Therefore, each of the disks contains exactly one zero of $f$.
This complete the proof of the theorem.
\qedhere
\end{step} 
\end{proof}

%
%

\section{Semilocal convergence of Ehrlich's method}
\label{sec:Semilocal-convergence-of-Ehrlich-method}

In this section, we obtain a new semilocal convergence theorem with error estimates for Ehrlich method \eqref{eq:Ehrlich-iteration}. 
Our theorem generalizes and improves the results of 
Petkovi\'c \cite{Pet96}, 
Petkovi\'c and Ili\'c \cite{PI97},
Petkovi\'c and Herceg \cite{PH97}, 
Zheng and Huang \cite{ZH00}, 
Petkovi\'c and Herceg \cite{PH01} and 
Proinov \cite{Pro06c}. 
This section can be considered as a continuation of our paper \cite{Pro15c}, where we provide a detailed local convergence 
analysis of the Ehrlich method \eqref{eq:Ehrlich-iteration}. 

\medskip
In this section, we continue to use the notations $\sigma_i(x)$ and $\hat{\sigma}_i(x)$ defined by \eqref{eq:sigma}, 
but now ${\hat{x} = \Phi(x)}$.

\begin{lem} \label{lem:Ehrlich-1}
Let ${f \in \Kset[z]}$ be a  polynomial of degree ${n \ge 2}$, ${1 \le p \le \infty}$ and ${x \in \mathcal{D}}$ 
be a vector such that
\begin{equation} \label{eq:Ehrlich-1-initial-condition}
E_f(x) < 1 / a,
\end{equation}
where ${a = (n - 1)^{1/q}}$.
Then ${x \in \mathscr{D}}$ and
\begin{equation} \label{eq:Ehrlich-1}
\|x - \Phi(x)\| \preceq \gamma(E_f(x)) \, \|W_f(x)\|,
\end{equation}
where the real function $\gamma$ is defined by
\begin{equation} \label{eq:gamma-Ehrlich}
\gamma(t) = 1 / (1 - a t) \, .
\end{equation}
\end{lem}

\begin{proof}
From Theorem~\ref{thm:iteration-function-properties}(v) and \eqref{eq:Ehrlich-1-initial-condition}, we obtain
\begin{equation} \label{eq:sigma-estimate-2}
|1 + \sigma_i(x)| \ge 1 - |\sigma_i(x)| \ge 1 - a E_f(x) > 0.
\end{equation}
According to \eqref{eq:Ehrlich-iteration-function-domain} this means that ${x \in \mathscr{D}}$. 
From \eqref{eq:Ehrlich-iteration-function} and \eqref{eq:sigma-estimate-2}, we get
\[
|x_i - \Phi_i(x)| = \frac{|W_i(x)|}{| 1 + \sigma_i(x) |} \le \frac{|W_i(x)|}{1 - a E_f(x)} = \gamma(E_f(x)) \, |W_i(x)|
\]
which implies \eqref{eq:Ehrlich-1}.
\end{proof}
 
According to \eqref{eq:main-psi-mu}, we define the functions $\psi$ and $\mu$ as follows:
\begin{equation} \label{eq:psi-mu-Ehrlich}
\psi(t) = 1 - b t \gamma(t) = \frac{1 - (a + b) t}{1 - a t} \quad\text{and}\quad
\mu(t) = 1 - t \gamma(t) = \frac{1 - (a + 1) t}{1 - a t} \, ,
\end{equation}
where $\gamma$ is defined by \eqref{eq:gamma-Ehrlich}. Note that if ${0 \le t < 1 / (a + b)}$, then ${\psi(t) > 0}$.

\begin{lem} \label{lem:Ehrlich-2}
Let ${f \in \Kset[z]}$ be a  polynomial of degree ${n \ge 2}$, ${1 \le p \le \infty}$ and ${x \in \mathcal{D}}$ be a vector such that
\begin{equation} \label{eq:Ehrlich-2-initial-condition}
E_f(x) < 1 / (a + b).
\end{equation}
Then ${x \in \mathscr{D}}$, ${\Phi(x) \in \mathcal{D}}$ and for each ${i \in I_n \, }$,
\begin{equation} \label{eq:Ehrlich-2}
W_i(\hat{x}) = (\hat{x}_i - x_i) (\hat{\sigma}_i(x) - \sigma_i(x)) \prod_{j \ne i} {\frac{\hat{x}_i - x_j}{\hat{x}_i - \hat{x}_j}} \, ,
\end{equation}
where ${\hat{x} = \Phi(x)}$.
\end{lem}

\begin{proof}
It follows from Lemma \ref{lem:Ehrlich-1} and Theorem~\ref{thm:iteration-function-properties}(i) 
that ${x \in \mathscr{D}}$ and ${\Phi(x) \in \mathcal{D}}$.   
Hence, both sides of \eqref{eq:Ehrlich-2} are well-defined. 
Now we shall prove formula \eqref{eq:Ehrlich-2}.
Without loss of generality, we may assume that $f$ is a monic polynomial.
Then by Proposition~8.1 of \cite{Pro15a} with ${z = \hat{x}_i}$, we obtain
\begin{equation} \label{eq:Lagrange-formula-Ehrlich}
f(\hat{x}_i) = (\hat{x}_i - x_i) \left( 1 + \frac{W_i(x)}{\hat{x}_i - x_i} + \hat{\sigma}_i(x) \right) \prod_{j \ne i}{(\hat{x}_i - x_j)}.
\end{equation}
It follows from the definition of $\hat{x}$ and \eqref{eq:Ehrlich-iteration-function} that 
\[
\hat{x}_i = x_i  - W_i(x) \left( 1 + \sigma_i(x) \right)^{-1}.
\]
This equality can be written in the form
\begin{equation} \label{eq:BS-IF-form}
\frac{W_i(x)}{\hat{x}_i - x_i} =  - 1 - \sigma_i(x).
\end{equation}
From \eqref{eq:Lagrange-formula-Ehrlich} and \eqref{eq:BS-IF-form}, we get
\[
f(\hat{x}_i) = (\hat{x}_i - x_i) (\hat{\sigma}_i(x) - \sigma_i(x)) \prod_{j \ne i}{(\hat{x}_i - x_j)}.
\]
Dividing both sides of this equality by ${\prod_{j \ne i}{(\hat{x}_i - \hat{x}_j)}}$, we get \eqref{eq:Ehrlich-2} which completes the proof. 
\end{proof}

\begin{lem} \label{lem:Ehrlich-3}
Under the assumptions of the previous lemma,
\begin{equation} \label{eq:Ehrlich-3}
\|W_f(\Phi(x))\| \preceq \beta(E_f(x)) \, \|W_f(x)\|
\end{equation}
where the real function $\beta$ is defined as follows
\begin{equation} \label{eq:beta-Ehrlich}
\beta(t) = \frac{a t^2}{(1 - a t) (1 - (a + 1) t)} \left (1 + \frac{a t}{(n - 1) (1 - (a + b) t)} \right )^{n - 1},
\end{equation}
where ${a = (n - 1)^{1/q}}$ and ${b = 2^{1/q}}$. 
\end{lem}

\begin{proof}
Let $i \in I_n$ be fixed.
It follows from Lemma~\ref{lem:Ehrlich-2} that
\begin{equation} \label{eq:Ehrlich-3a}
	W_i(\Phi(x)) =  A_i(x) B_i(x) (\hat{x}_i - x_i) \, , 
\end{equation}
where ${A_i(x)}$ and ${B_i(x)}$ are defined by \eqref{eq:A-B-definition}.
From Theorem~\ref{thm:iteration-function-properties}(vi) and Theorem~\ref{thm:iteration-function-properties}(vii), we have
\begin{equation} \label{eq:A-B-estimates}
|A_i(x)| \le \frac{a \, \gamma(E_f(x)) \, E_f(x)^2}{\mu(E_f(x))} \,\, \text{ and } \,\,
|B_i(x)| \le \left( 1 + \frac{a E_(x) \, \gamma(E_(x))}{(n - 1) \psi(E_f(x))} \right)^{n-1},
\end{equation}
where $\psi$ and $\mu$ are defined by \eqref{eq:psi-mu-Ehrlich}. It follows from Lemma~\ref{lem:Ehrlich-1} that
\begin{equation} \label{eq:Ehrlich-1a}
|\hat{x}_i - x_i| \le \gamma(E_f(x)) \, |W_i(x)|.
\end{equation}
From \eqref{eq:Ehrlich-3a}, \eqref{eq:A-B-estimates} and \eqref{eq:Ehrlich-1a}, we get
\begin{equation} \label{eq:Ehrlich-4a}
|W_i(\Phi(x))| \le \beta(E_f(x)) \, |W_i(x)|
\end{equation}
which yields the inequality \eqref{eq:Ehrlich-3}. 
\end{proof}

We define the function ${\phi = \beta / \psi}$, where $\psi$ and $\beta$ are define by 
\eqref{eq:psi-mu-Ehrlich} and \eqref{eq:beta-Ehrlich}, respectively. It is easy to calculate that
\begin{equation} \label{eq:phi-Ehrlich}
\phi(t) = \frac{a t^2}{(1 - (a + 1) t) (1 - (a + b) t)} \left (1 + \frac{a t}{(n - 1) (1 - (a + b) t)} \right )^{n - 1}.
\end{equation}

Now we are ready to state the main result of this section which generalizes, improves and complements 
all previous results in this area.

\begin{thm} 
\label{thm:Ehrlich-semilocal}
Let ${(\Kset,|\cdot|)}$ be a complete normed field,
$f \in \Kset[z]$ be a polynomial of degree $n \ge 2$ and $1 \le p \le \infty$.
Suppose $x^{(0)} \in \Kset^n$ is an initial guess with distinct components satisfying
\begin{equation} \label{eq:Ehrlich-semilocal-initial-condition}
E_f(x^{(0)}) < 1 / (a + b) \quad\text{and}\quad \phi(E_f(x^{(0)})) \le 1,
\end{equation}
where the function $E_f$ is defined by \eqref{eq:FIC3-SM}, 
${a = (n - 1)^{1/q}}$, ${b = 2^{1/q}}$ and the function $\phi$ is defined by \eqref{eq:phi-Ehrlich}.
Then the following statements hold true:
\begin{enumerate}[(i)]
\item 
\textsc{Convergence}. Starting from $x^{(0)}$, Ehrlich iteration \eqref{eq:Ehrlich-iteration} is well-defined, remains in the closed ball 
${\overline{U}(x_0,\rho)}$ and converges to a root vector $\xi$ of $f$, where 
\begin{equation} \label{eq:inclusion-radius-Ehrlich}
	\rho =  \frac{\gamma(E_f(x^{(0)}))}{1 - \beta(E_f(x^{(0)}))} \, \| W_f(x^{(0)})\|
\end{equation}
and the functions $\beta$ and $\gamma$ are defined by \eqref{eq:beta-Ehrlich} and \eqref{eq:gamma-Ehrlich}, respectively.
Besides, the convergence is cubic provided that ${\phi(E_f(x^{(0)})) < 1}$.
\item 
\textsc{A priori estimate}. For all $k \ge 0$ we have the following error estimate
\begin{equation}  \label{eq:Ehrlich-semilocal-priori}
\left \| x^{(k)} - \xi \right \| \preceq A_k \, \frac{{\theta}^k {\lambda }^{(3^k - 1) / 2}}
{1 - \theta {\lambda} ^{3^k } } \left \| W_f(x^{(0)}) \right \| ,
\end{equation}
where $\lambda  = \phi(E_f(x^{(0)}))$, $\theta  = \psi(E_f(x^{(0)}))$,
$A_k = \gamma(E_f(x^{(0)}) \lambda^{(3^k - 1) / 2})$
and the function $\psi$ is defined by \eqref{eq:psi-mu-Ehrlich}.
\item 
\textsc{First a posteriori estimate}. For all $k \ge 0$ we have the following error estimate
\begin{equation}  \label{eq:Ehrlich-semilocal-posteriori-1}
\left\| x^{(k)} - \xi \right\| \preceq \frac{\gamma(E_f(x^{(k)}))}{1 - \beta(E_f(x^{(k)}))} \, \| W_f(x^{(k)})\|.
\end{equation}
\item 
\textsc{Second a posteriori estimate}. For all $k \ge 0$ we have the following error estimate
\begin{equation}  \label{eq:Ehrlich-semilocal-posteriori-2}
\left\| x^{(k + 1)} - \xi \right\| \preceq 
\frac{\theta_k \lambda_k}{1 - \theta_k (\lambda_k)^3} \, \gamma(E_f(x^{(k + 1)})) \left\| W_f(x^{(k)}) \right\| ,
\end{equation}
where $\lambda_k = \phi(E_f(x^{(k)}))$ and  $\theta_k = \psi(E_f(x^{(k)}))$.
\item \textsc{Localization of the zeros}.
If ${\phi(E_f(x^{(0)} )) < 1}$, then $f$ has $n$ simple zeros in $\Kset$ and for every
${k \ge 0}$ the closed disks
\begin{equation}  \label{eq:disks-disjoint-Ehrlich}
D_i^k = \{z \in \Kset : |z-x_i^{(k)}| \le r_i^{(k)} \},
\quad i = 1,2,\ldots,n,
\end{equation}
where 
\[
r_i^{(k)} = \frac{\gamma(E_f(x^{(k)}))}{1 - \beta(E_f(x^{(k)}))} \, |W_i(x^{(k)})|,
\]
are mutually disjoint and each of them contains exactly one zero of $f$.	
\end{enumerate}
\end{thm}

\begin{proof}
It follows immediately from Theorem~\ref{thm:general-semilocal-convergence-theorem-SM} and 
Lemmas \ref{lem:Ehrlich-1} and \ref{lem:Ehrlich-3}.
\end{proof}

\begin{rem}
In our work \cite{Pro06c} we have stated without proof a weaker version of Theorem~\ref{thm:Ehrlich-semilocal},  
which generalizes and improves the results in \cite{Pet96,PI97,PH97,PH01}.
It should be noted that the corollaries 3.1, 3.2, 3.3 and 3.4 given in \cite{Pro06c} can be improved using 
Theorem~\ref{thm:Ehrlich-semilocal} instead of Theorem~3.1 of \cite{Pro06c}.  
We end this section with a result which improves Corollary~3.2 of \cite{Pro06c}.
This result generalizes and improves Theorem~1 of Zheng and Huang \cite{ZH00} as well as the results \cite{Pet96,PI97,PH97}.     
\end{rem} 

\begin{cor} \label{cor:semilocal-convergence-Ehrlich-2}
Let ${(\Kset,|\cdot|)}$ be a complete normed field,
$f \in \Kset[z]$ be a polynomial of degree $n \ge 2$ and $1 \le p \le \infty$.
Suppose $x^{(0)} \in \Kset^n$ is an initial guess with distinct components satisfying
\begin{equation} \label{eq:Ehrlich-semilocal-initial-condition-cor}
\left\| \frac{W_f(x^{(0)})}{d(x^{(0)})} \right\|_p \le \frac{1}{2 a + 2} \, ,
\end{equation}
where ${a = (n - 1)^{1/q}}$. In the case ${n = 2}$ and ${p = \infty}$, we assume that inequality 
\eqref{eq:Ehrlich-semilocal-initial-condition-cor} is strict.
Then $f$ has $n$ simple zeros in $\Kset$, 
the Ehrlich iteration \eqref{eq:Ehrlich-iteration} is well-defined, lies in the closed ball ${\overline{U}(x_0,\rho)}$ with radius 
$\rho$ defined by \eqref{eq:inclusion-radius-Ehrlich}, converges cubically to $\xi$ with error estimates 
\eqref{eq:Ehrlich-semilocal-priori}, 
\eqref{eq:Ehrlich-semilocal-posteriori-1} and \eqref{eq:Ehrlich-semilocal-posteriori-2}, 
and for every ${k \ge 0}$ the closed disks \eqref{eq:disks-disjoint-Ehrlich} are mutually disjoint and each of them contains exactly one zero of $f$. 
\end{cor}

\begin{proof}
Let ${R = 1 / (2 a + 2)}$. By Theorem~\ref{thm:Ehrlich-semilocal} it suffices to prove that ${\phi(R) \le 1}$ 
with equality only if ${n = 2}$ and ${p = \infty}$. 
It is easy to compute that
\begin{equation} \label{eq:phi(R)}
	\phi(R) = \frac{a}{(a + 1)(a + 2 - b)} \left (1+ \frac{a}{(n - 1) (a + 2 - b)} \right)^{n-1}. 
\end{equation} 
Note that ${a \ge 1}$ and ${1 \le b \le 2}$.
Using the inequality ${a / (a + 2 - b) \le 1}$, we obtain
\begin{equation} \label{eq:phi(R)-estimates}
\phi(R) < \frac{e}{a + 1} \quad\text{and}\quad \phi(R) < \frac{a e}{(a + 1)(a + 2 - b)} \, .
\end{equation}
If $a \ge e - 1$, then from the first estimate in \eqref{eq:phi(R)-estimates}, we get ${\phi(R) < 1}$.
If $a < e - 1$ and ${n \ge 3}$, then from the inequality ${a \le b}$ and the second estimate in \eqref{eq:phi(R)-estimates}, we get 
\(
{\phi(R) < a e / (2 a + 2) \le (e - 1) / 2 < 1}.
\)
In the case ${n = 2}$, we have ${\phi(R) = (2 - b / 2) / (3 - b)^2} \le 1$ with equality if and only if ${p = \infty}$.
This completes the proof of the corollary.
\end{proof}

%
%

\section{Semilocal convergence of Dochev-Byrnev's method}
\label{sec:Semilocal-convergence-of-Dochev-Byrnev-method}

In this section, we first show that both methods \eqref{eq:Dochev-Byrnev-iteration} and \eqref{eq:Tanabe-iteration} are identical.
Second, we provide a new semilocal convergence theorem with error estimates for the Dochev-Byrnev method. 
Our theorem generalizes and improves the result of
Ili\'c and Herceg \cite{IH98} and 
Petkovi\'c, Herceg and Ili\'c \cite{PHI97,PHI98}.

\begin{thm} \label{thm:Dochev-Byrnev-Tanabe-equivalence}
The Dochev-Byrnev method \eqref{eq:Dochev-Byrnev-iteration} is identical with the Pre{\v s}i\'c-Tanabe method \eqref{eq:Tanabe-iteration}, that is,
\begin{equation} \label{eq:Dochev-Byrnev=Tanabe}
	\mathscr{F}(x) = \mathscr{T}(x) \quad\text{for all } \, x \in \mathcal{D}.
\end{equation}
where the iteration functions $\mathscr{F}$ and $\mathscr{T}$ are defined by 
\eqref{eq:Dochev-Byrnev-iteration-function} and \eqref{eq:Tanabe-iteration-function}, respectively.
\end{thm}

\begin{proof}
Without loss of generality we may assume that $f$ is monic.
Let ${x \in \mathcal{D}}$ be fixed. 
It is easy to see that for every ${i \in I_n}$,
\[
g'(x_i) = \prod_{j \ne i} {(x_i - x_j)}, \quad \frac{f(x_i)}{g'(x_i)} = W_i(x) \quad\text{and}\quad 
\frac{g''(x_i)}{g'(x_i)} = 2 \sum_{j \ne i} {\frac{1}{x_i - x_j}} \, ,
\]
where the polynomial $g$ is defined by \eqref{eq:polynomial-g} with ${C_0 = 1}$.
Using these equalities, we can write \eqref{eq:Dochev-Byrnev-iteration-function} in the form
\[
\mathscr{F}_i(x) = x_i - W_i(x) \left( 2 - \frac{f'(x_i)}{\prod_{j \ne i} {(x_i - x_j)}} + \sum_{j \ne i} {\frac{W_i(x)}{x_i - x_j}}
\right).
\]
It follows from this and \eqref{eq:f-derivative} that 
\[
\mathscr{F}_i(x) = x_i - W_i(x) \left( 1 - \sum_{j \ne i} {\frac{W_j(x)}{x_i - x_j}}
\right).
\]
From this and \eqref{eq:Tanabe-iteration-function}, we obtain \eqref{eq:Dochev-Byrnev=Tanabe}.
This completes the proof.
\end{proof}

In this section, we again continue to use the notations $\sigma_i(x)$ and $\hat{\sigma}_i(x)$ defined by \eqref{eq:sigma}, 
but now ${\hat{x} = \mathscr{F}(x) = \mathscr{T}(x)}$.

\begin{lem} \label{lem:Tanabe-1}
Let ${f \in \Kset[z]}$ be a  polynomial of degree ${n \ge 2}$, ${1 \le p \le \infty}$ and ${x \in \mathcal{D}}$. 
Then 
\begin{equation} \label{eq:Tanabe-1}
\|x - \mathscr{F}(x)\| \preceq \gamma(E_f(x)) \, \|W_f(x)\|,
\end{equation}
where the real function $\gamma$ is defined by
\begin{equation} \label{eq:gamma-Tanabe}
\gamma(t) = 1 + a t.
\end{equation}
\end{lem}

\begin{proof}
It follows from the definition of $\hat{x}$, \eqref{eq:Dochev-Byrnev=Tanabe} and \eqref{eq:Tanabe-iteration-function} that 
\begin{equation} \label{eq:eq:Tanabe-iteration-function-1}
\hat{x}_i = x_i  - W_i(x) ( 1 - \sigma_i(x)).
\end{equation}
From \eqref{eq:eq:Tanabe-iteration-function-1} and \eqref{eq:sigma}, we obtain
\[
|x_i - \mathscr{F}_i(x)| = |(1 - \sigma_i(x)) \, W_i(x)| \le (1 + |\sigma_i(x)|) \, |W_i(x)| \le \gamma(E_f(x)) \, |W_i(x)|
\]
which implies \eqref{eq:Tanabe-1}.
\end{proof}
 
In accordance with \eqref{eq:main-psi-mu}, we define the functions $\psi$ and $\mu$ by
\begin{equation} \label{eq:psi-mu-Tanabe}
\psi(t) = 1 - b t \gamma(t) = 1 - b t - a b t^2 \quad\text{and}\quad \mu(t) = 1 - t \gamma(t) = 1 - t - a t^2,
\end{equation}
where $\gamma$ is defined by \eqref{eq:gamma-Tanabe}. Note that ${\psi(t) > 0}$ if ${0 \le t < 2 / (b + \sqrt{b^2 + 4 a b})}$.

\begin{lem} \label{lem:Tanabe-2}
Let ${f \in \Kset[z]}$ be a  polynomial of degree ${n \ge 2}$, ${1 \le p \le \infty}$ and ${x \in \mathcal{D}}$ be a vector such that
\begin{equation} \label{eq:Tanabe-3-initial-condition}
E_f(x) < \min \left\{ \frac{1}{a} , \frac{2}{b + \sqrt{b^2 + 4 a b}} \right\} \, .
\end{equation}
Then ${\mathscr{F}(x) \in \mathcal{D}}$ and for each ${i \in I_n \,}$,
\begin{equation} \label{eq:Tanabe-2}
W_i(\hat{x}) = \frac{(x_i - \hat{x}_i) (\sigma_i(x) - \hat{\sigma}_i(x) + \sigma_i(x) \, \hat{\sigma}_i(x))}{1 -\sigma_i(x)} 
\prod_{j \ne i} {\frac{\hat{x}_i - x_j}{\hat{x}_i - \hat{x}_j}} \, ,
\end{equation}
where ${\hat{x} = \mathscr{F}(x)}$.
\end{lem}

\begin{proof}
It follows from Lemma~\ref{lem:Tanabe-1} and Theorem~\ref{thm:iteration-function-properties}(i) that ${\mathscr{F}(x) \in \mathcal{D}}$.
Let ${i \in I_n}$ be fixed.
In view of \eqref{eq:sigma-estimate-2}, we have that ${\sigma_i(x) \ne -1}$.
Therefore, both sides of \eqref{eq:Tanabe-2} are well-defined. Now we shall prove the equality \eqref{eq:Tanabe-2}.
The equality \eqref{eq:Tanabe-iteration-function} can be written in the form
\begin{equation} \label{eq:Tanabe-IF-form}
\frac{W_i(x)}{\hat{x}_i - x_i} =  - \frac{1}{1 - \sigma_i(x)} \, .
\end{equation}
Without loss of generality, we may assume that $f$ is a monic polynomial.
From \eqref{eq:Lagrange-formula-Ehrlich} and \eqref{eq:Tanabe-IF-form}, we get
\[
f(\hat{x}_i) = \frac{((x_i - \hat{x}_i)) (\sigma_i(x) - \hat{\sigma}_i(x) + \sigma_i(x) \, \hat{\sigma}_i(x))}{1 -\sigma_i(x)} 
\prod_{j \ne i}{(\hat{x}_i - x_j)}.
\]
Dividing both sides of this equality by ${\prod_{j \ne i}{(\hat{x}_i - \hat{x}_j)}}$, we get \eqref{eq:Tanabe-2} which completes the proof. 
\end{proof}

\begin{lem} \label{lem:Tanabe-3}
Let ${f \in \Kset[z]}$ be a polynomial of degree ${n \ge 2}$, ${1 \le p \le \infty}$ and
${x \in \mathcal{D}}$ be a vector satisfying condition \eqref{eq:Tanabe-3-initial-condition}.
Then
\begin{equation} \label{eq:Tanabe-3}
\|W_f(\mathscr{F}(x))\| \preceq \beta(E_f(x)) \, \|W_f(x)\|,
\end{equation}
where the  real function $\beta$ is defined as follows
\begin{equation} \label{eq:beta-Tanabe}
\beta(t) = \frac{a t^2 (1 + a t) (1 + a + a t)}{(1 - a t) (1 - t - a t^2)} 
\left (1 + \frac{a t (1+ a t)}{(n - 1) (1 - b t - a b t^2)} \right )^{n - 1},
\end{equation}
where ${a = (n - 1)^{1/q}}$ and ${b = 2^{1/q}}$. 
\end{lem}

\begin{proof}
Let $i \in I_n$ be fixed.
It follows from Lemma~\ref{lem:Tanabe-2} that
\begin{equation} \label{eq:Tanabe-3a}
	W_i(\mathscr{F}(x)) = A_i(x) B_i(x) (x_i - \hat{x}_i), 
\end{equation}
where 
\[
A_i(x) = \frac{\sigma_i(x) - \hat{\sigma}_i(x) + \sigma_i(x) \, \hat{\sigma}_i(x)}{1 - \sigma_i(x)}
\quad\text{and}\quad 
B_i(x) = \prod_{j \ne i} {\frac{\hat{x}_i - x_j}{\hat{x}_i - \hat{x}_j}} \, . 
\]
From Theorem~\ref{thm:iteration-function-properties}(vi), estimate \eqref{eq:sigma} and 
Theorem~\ref{thm:iteration-function-properties}(vii), we obtain
\begin{equation} \label{eq:A-estimate-Tanabe}
|A_i(x)| \le \frac{|\sigma_i(x) - \hat{\sigma}_i(x)| + |\sigma_i(x)| \, |\hat{\sigma}_i(x)|}{1 - |\sigma_i(x)|}
         \le \frac{a E_f(x)^2 (\gamma(E_f(x)) + a)}{\mu(E_f(x)) (1 - a E_f(x))} \, , 
\end{equation}
where $\mu$ is defined by \eqref{eq:psi-mu-Tanabe}.
The rest of the proof is the same as the one of Lemma~\ref{lem:Ehrlich-3}, except that the first estimate in \eqref{eq:A-B-estimates} 
must be replaced by \eqref{eq:A-estimate-Tanabe}.
\end{proof}

We can define the function $\phi$ by
\begin{equation} \label{eq:phi-Tanabe}
\phi(t) = \frac{\beta(t)}{\psi(t)} = 
\frac{a t^2 (1 + a t) (1 + a + a t)}{(1 - a t) (1 - t - a t^2) (1 - b t - a b t^2)} 
\left (1 + \frac{a t (1+ a t)}{(n - 1) (1 - b t - a b t^2)} \right )^{n - 1},
\end{equation}
where $\psi$ and $\beta$ are define by \eqref{eq:psi-mu-Tanabe} and \eqref{eq:beta-Tanabe}, respectively.

\medskip
Now we are ready to state the main result of this section. 

\begin{thm} 
\label{thm:Tanabe-semilocal}
Let ${(\Kset,|\cdot|)}$ be a complete normed field,
$f \in \Kset[z]$ be a polynomial of degree $n \ge 2$ and $1 \le p \le \infty$.
Suppose $x^{(0)} \in \Kset^n$ is an initial guess with distinct components satisfying
\begin{equation} \label{eq:Tanabe-semilocal-initial-condition}
E_f(x^{(0)}) < \min \left\{ \frac{1}{a} , \frac{2}{b + \sqrt{b^2 + 4 a b}} \right\}
\quad\text{and}\quad \phi(E_f(x^{(0)})) \le 1,
\end{equation}
where the function $E_f$ is defined by \eqref{eq:FIC3-SM}, ${a = (n - 1)^{1/q}}$, ${b = 2^{1/q}}$ and the function $\phi$ is defined by \eqref{eq:phi-Tanabe}.
Then the following statements hold true:
\begin{enumerate}[(i)] 
\item 
\textsc{Convergence}. Starting from $x^{(0)}$, Dochev-Byrnev iteration \eqref{eq:Dochev-Byrnev-iteration} is well-defined, remains in the closed ball ${\overline{U}(x_0,\rho)}$ and converges to a root vector $\xi$ of $f$, where 
\begin{equation} \label{eq:inclusion-radius-Tanabe}
	\rho =  \frac{\gamma(E_f(x^{(0)}))}{1 - \beta(E_f(x^{(0)}))} \, \| W_f(x^{(0)})\|
\end{equation}
and the functions $\beta$ and $\gamma$ are defined by \eqref{eq:beta-Tanabe} and \eqref{eq:gamma-Tanabe}, respectively.
Besides, the convergence is cubic provided that ${\phi(E_f(x^{(0)})) < 1}$.
\item 
\textsc{A priori estimate}. For all $k \ge 0$ we have the following error estimate
\begin{equation}  \label{eq:Tanabe-semilocal-priori}
\left \| x^{(k)} - \xi \right \| \preceq A_k \, \frac{{\theta}^k {\lambda }^{(3^k - 1) / 2}}
{1 - \theta {\lambda}^{3^k } } \left \| W_f(x^{(0)}) \right \| ,
\end{equation}
where $\lambda  = \phi(E_f(x^{(0)}))$, $\theta  = \psi(E_f(x^{(0)}))$,
$A_k = \gamma(E_f(x^{(0)}) \lambda^{(3^k - 1) / 2})$
and the real  function $\psi$ is defined by \eqref{eq:psi-mu-Tanabe}.
\item 
\textsc{First a posteriori estimate}. For all $k \ge 0$ we have the following error estimate
\begin{equation}  \label{eq:Tanabe-semilocal-posteriori-1}
\left\| x^{(k)} - \xi \right\| \preceq \frac{\gamma(E_f(x^{(k)}))}{1 - \beta(E_f(x^{(k)}))} \, \| W_f(x^{(k)})\|.
\end{equation}
\item 
\textsc{Second a posteriori estimate}. For all $k \ge 0$ we have the following error estimate
\begin{equation}  \label{eq:Tanabe-semilocal-posteriori-2}
\left\| x^{(k + 1)} - \xi \right\| \preceq 
\frac{\theta_k \lambda_k}{1 - \theta_k (\lambda_k)^3} \, \gamma(E_f(x^{(k + 1)})) \left\| W_f(x^{(k)}) \right\| ,
\end{equation}
where $\lambda_k = \phi(E_f(x^{(k)}))$ and  $\theta_k = \psi(E_f(x^{(k)}))$.
\item \textsc{Localization of the zeros}.
If ${\phi(E_f(x^{(0)} )) < 1}$, then $f$ has $n$ simple zeros in $\Kset$ and for every
${k \ge 0}$ the closed disks
\begin{equation}  \label{eq:disks-disjoint-Tanabe}
D_i^k = \{z \in \Kset : |z-x_i^{(k)}| \le r_i^{(k)} \},
\quad i = 1,2,\ldots,n,
\end{equation}
where 
\[
r_i^{(k)} = \frac{\gamma(E_f(x^{(k)}))}{1 - \beta(E_f(x^{(k)}))} \, |W_i(x^{(k)})|,
\]
are mutually disjoint and each of them contains exactly one zero of $f$.	
\end{enumerate}
\end{thm}

\begin{proof}
It follows immediately from Theorem~\ref{thm:general-semilocal-convergence-theorem-SM} and 
Lemmas \ref{lem:Tanabe-1} and \ref{lem:Tanabe-3}.
\end{proof}

The following result improves and complements the result of Petkovi\'c, Herceg and Ili\'c \cite{IH98,PHI97,PHI98} in several directions.

\begin{cor} \label{cor:semilocal-convergence-Tanabe-infinity}
Let ${(\Kset,|\cdot|)}$ be a complete normed field, and let $f \in \Kset[z]$ be a polynomial of degree $n \ge 2$.
Suppose $x^{(0)} \in \Kset^n$ is an initial guess with distinct components satisfying
\[
E_f(x^{(0)}) = \left\| \frac{W(x^{(0)})}{d(x^{(0)})} \right\|_\infty \le \frac{4}{9 n} \, .
\]
Then $f$ has $n$ simple zeros in $\Kset$, 
the Dochev-Byrnev iteration \eqref{eq:Dochev-Byrnev-iteration} is well-defined, lies in the closed ball ${\overline{U}(x_0,\rho)}$ 
with radius $\rho$ defined by \eqref{eq:inclusion-radius-Tanabe}, converges cubically to $\xi$ with error estimates \eqref{eq:Tanabe-semilocal-priori}, \eqref{eq:Tanabe-semilocal-posteriori-1} and \eqref{eq:Tanabe-semilocal-posteriori-2}, 
and for every ${k \ge 0}$ the closed disks \eqref{eq:disks-disjoint-Tanabe} are mutually disjoint and each of them contains exactly one zero of $f$.
\end{cor}

\begin{proof}
Define the sequence $(\phi_n)_{n=2}^\infty$ by ${\phi_n = \phi(4 / (9 n))}$. 
It can easily be proved that ${\phi_n < 1}$ for every ${n \ge 2}$.
Then the proof follows from Theorem~\ref{thm:Tanabe-semilocal}. 
\end{proof}

\begin{cor} \label{cor:semilocal-convergence-Tanabe-1}
Let ${(\Kset,|\cdot|)}$ be a complete normed field, and let $f \in \Kset[z]$ be a polynomial of degree $n \ge 2$.
Suppose $x^{(0)} \in \Kset^n$ is an initial guess with distinct components satisfying
\[
E_f(x^{(0)}) = \left\| \frac{W(x^{(0)})}{d(x^{(0)})} \right\|_1 \le R = 0.2636\ldots.
\]
where $R$ is the unique solution of the equation
\begin{equation} \label{eq:equation}
\frac{t^2 (1 + t) (2 + t)}{(1 - t) (1 - t - t^2)^2} \, \exp\left(\frac{t + t^2}{1 - t - t^2}\right) = 1 
\end{equation}
in the interval ${(0,(\sqrt{5} - 1) / 2)}$.
Then $f$ has $n$ simple zeros in $\Kset$, 
the Dochev-Byrnev iteration \eqref{eq:Dochev-Byrnev-iteration} is well-defined, lies in the closed ball ${\overline{U}(x_0,\rho)}$ with radius $\rho$ defined by \eqref{eq:inclusion-radius-Tanabe}, converges cubically to $\xi$ with error estimates \eqref{eq:Tanabe-semilocal-priori}, \eqref{eq:Tanabe-semilocal-posteriori-1} and \eqref{eq:Tanabe-semilocal-posteriori-2}, 
and for every ${k \ge 0}$ the closed disks \eqref{eq:disks-disjoint-Tanabe} are mutually disjoint and each of them contains 
exactly one zero of $f$.
\end{cor}

\begin{proof}
It is easy to show that ${\phi(t) < g(t)}$ for ${0 < t < (\sqrt{5} - 1) / 2}$, where $g(t)$ denotes the left-hand side of the equation 
\eqref{eq:equation}. Therefore, ${\phi(R) < 1}$ which according to Theorem~\ref{thm:Tanabe-semilocal} completes the proof.
\end{proof}

%
%

\end{document}